\documentclass{amsart}

\usepackage{amsmath, amssymb, amsthm} 
\usepackage{ascmac,fancybox}
\usepackage{braket, ulem}
\usepackage{enumerate}
\usepackage{graphicx}

\newtheorem{thm}{Theorem}[section]

\newtheorem{lem}[thm]{Lemma}

\newtheorem{prop}[thm]{Proposition}
\newtheorem{clm}{Claim}

\theoremstyle{definition}
\newtheorem{dfn}{Definition}[section]
\newtheorem{rmk}[dfn]{Remark}

\renewcommand{\mod}{\operatorname{mod}}
\numberwithin{equation}{section}

\title[Geometric intersection number and symplectic expansion]{%
Geometric intersection number of simple closed curves on a surface and symplectic expansions of free groups}
\author{Ryosuke YAMAMOTO}
\address{Faculty of Education, Gunma University, 4-2 Aramaki-machi, Maebashi, Gunma, 371-8510 Japan}
\email{yamryo0202@gunma-u.ac.jp}

\begin{document}

\begin{abstract}
 For two oriented simple closed curves on a compact orientable surface
 with a connected boundary
 we introduce
 a simple computation of a value in the first homology group of the surface,
 which detects in some cases that
 the geometric intersection number of the curves is greater than zero
 when their algebraic intersection number is zero.
 The value,
 computed from two elements of the fundamental group of the surface
 corresponding to the curves,
 is found in the difference between one of the elements
 and its image of the action of Dehn twist along the other.
 To give a description of the difference
 symplectic expansions of free groups is an effective tool,
 since we have an explicit formula
 for the action of Dehn twist on the target space of the expansion
 due to N.\ Kawazumi and Y.\ Kuno.
\end{abstract}

\maketitle


\section{Introduction}\label{sec:intro}

 Let $\Sigma=\Sigma_{g,1}$ be a compact orientable surface of genus $g \geq 1$ with a connected boundary.
%
 Let $\alpha$ and $\beta$ be oriented simple closed curves on $\Sigma$.
 We always assume that the intersections of the curves are transverse double points.
 The geometric intersection number of $\alpha$ with $\beta$,
 we denote by $i_{G}(\alpha, \beta)$,
 is the minimal number of intersection points of $\alpha$ with
 any simple closed curve on $\Sigma$ freely homotopic to $\beta$.
 The algebraic intersection number of $\alpha$ with $\beta$,
 we denote by $i_{A}(\alpha, \beta)$,
 is the sum of signs of intersection points of $\alpha$ with $\beta$,
 where the sign of an intersection point of $\alpha$ with $\beta$ is $+1$
 when a pair of tangent vectors of $\alpha$ and $\beta$ in this order
 is consistent with an oriented basis for $\Sigma$,
 otherwise the sign is $-1$.

 In this paper
 we describe an algebraic value associating with two simple closed curves $\alpha$ and $\beta$ on
 $\Sigma$,
 which may detect that
 $i_{G}(\alpha, \beta)$
 is greater than zero
 when
 $i_{A}(\alpha, \beta)$
 is zero.
 %
 Our strategy 
 is as follows:
%
 The geometric intersection number of two simple closed curves on a surface
 is closely related to
 the change of the homotopy type of one of the curves
 by performing a Dehn twist along the other.
 In particular
 it is known (e.g.\ \cite{FM}) that
 $i_{G}(\alpha, \beta)=0$
 if and only if
 the Dehn twist along one of the curves
 does not change homotopy type of the other,
 in other words,
 there exists a based homotopy class $b$ of a based loop freely homotopic to $\beta$
 such that $t_{\alpha}(b)=b$.
 %
 To analyze the action of Dehn twist on the fundamental group of a surface,
 the symplectic expansion,
 a certain type of (generalized) Magnus expansion defined by G.\ Massuyeau \cite{M},
 of free groups is an effective tool,
 since we have an explicit formula
 describing the action of Dehn twist on the target space of the expansion
 due to N.\ Kawazumi and Y.\ Kuno. (See \S 2 for detail.)
 %
  Applying the formula we can compute differences between expansion of $t_{\alpha}(b)$ and of $b$.
 Since the first degree part of the expansions of $t_{\alpha}(b)$ and $b$ are coincide
 because of the condition $i_{A}(\alpha, \beta)=0$,
 we will focus on the second degree part of them.


 In the rest of this section we mention our main theorem and demonstrate an application of the
 theorem.
 In \S \ref{sec:prelim}, we recall the terminology of Kawazumi-Kuno's theory developed in \cite{KK},
 and prepare some propositions for use in \S \ref{sec:proof},
 in which we give a proof of the main theorem.

\subsection{Main result}\label{subsec:main_theorem}
 The fundamental group $\pi=\pi_{1}(\Sigma, p)$ of $\Sigma$ with a point $p$ on $\partial\Sigma$
 is a free group of rank $2g$
 and the first homology group of $\Sigma$ with coefficients in $\mathbb{Q}$,
 $H=H_{1}(\Sigma,\mathbb{Q})=\pi/[\pi,\pi]\otimes\mathbb{Q}$,
 is a free abelian group of rank $2g$.
 For $x \in \pi$, we denote by $|x|$
 the element of $H$ corresponding to $x$ via the abelianization of $\pi$.
%

 Let $\{x_{1},y_{1},\dots, x_{g},y_{g}\}$ be a symplectic generators of $\pi$
 as in Figure \ref{fig:symp_gene}.
  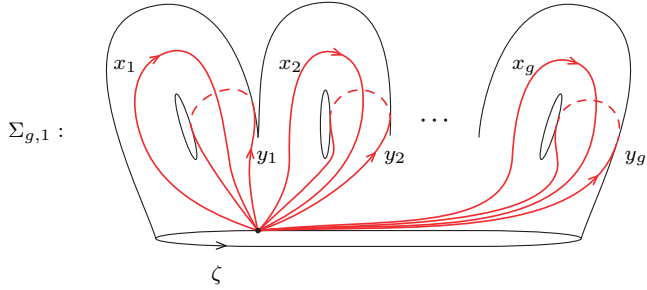
\begin{figure}[h]
   \begin{center}\unitlength=1mm
    \begin{picture}(100,35)(-10,0)
     \footnotesize
     %
     \put(0, 15){$\Sigma_{g,1}:$}
     \put(13,0){\includegraphics[height=33mm]{%
     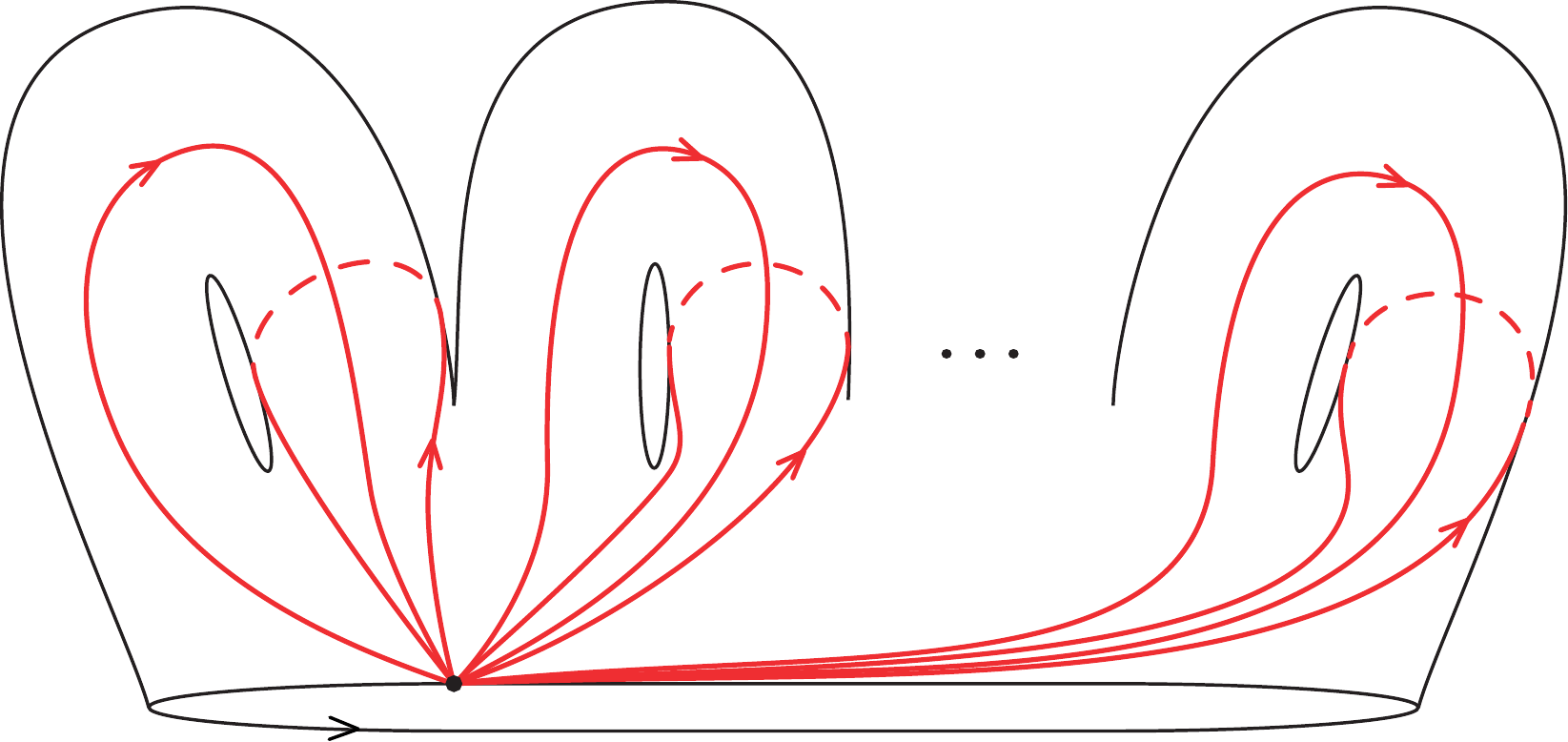}}
     \put(14, 24){$x_{1}$}\put(33, 12){$y_{1}$}
     \put(36, 24){$x_{2}$}\put(50, 12){$y_{2}$}
     \put(67, 24){$x_{g}$}\put(82, 12){$y_{g}$}
     \put(27, -4){$\zeta$}
    \end{picture}
   \end{center}
   \caption{symplectic generators of $\pi$.}
   \label{fig:symp_gene}
  \end{figure}
 %
 Putting $X_{i}=|x_{i}|$, $Y_{i}=|y_{i}|$ ($i=1,2,\dots,g$),
 we have a symplectic basis $\{X_{1},Y_{1},\dots, X_{g},Y_{g}\}$ of $H$,
 namely the basis satisfies
  $X_{i}{\cdot}Y_{i} = \delta_{ij}$,
  $X_{i}{\cdot}X_{j} = Y_{i}{\cdot}Y_{j} = 0$
  for $1 \leq i,j \leq g$,
 where
 $\cdot:H \times H \rightarrow \mathbb{Q}$ is the intersection form on $H$.

 With respect to the symplectic generators of $\pi$,
 we define a map $\ell:\pi \rightarrow H\wedge H$ by the following three rules:
 \begin{enumerate}[i)]
   \item $\ell(1)=0$,
   \item $\ell(x_{i})=\frac{1}{2}X_{i}\wedge Y_{i}$,
	 $\ell(y_{i})=-\frac{1}{2}X_{i}\wedge Y_{i}$, ($i=1,2,\dots,g$),
   \item For $\forall g,h \in \pi$, $\ell(gh)=\ell(g)+\ell(h)+\frac{1}{2}|g|\wedge |h|$.
 \end{enumerate}
 Note that we obtain $\ell(g^{-1})=-\ell(g)$ for $\forall g \in \pi$ from i) and iii).

 We consider that $H\wedge H$ acts on $H$ as follows:
 For $X \wedge Y \in H\wedge H$ and $Z \in H$,
 \begin{equation*}\label{eqn:ell}
  (X\wedge Y)(Z) := (Z{\cdot}X)Y-(Z{\cdot}Y)X.
 \end{equation*}

%
 \begin{thm}\label{thm:main}
  Let $\alpha$ and $\beta$ be oriented simple closed curves on $\Sigma$
  with $i_{A}(\alpha, \beta)=0$.
  If
  $\ell(a)(|b|)+\ell(b)(|a|)$ is not an element in $\mathbb{Z}|a|+\mathbb{Z}|b| \subset H$,
  where $a$ and $b$ are based homotopy classes of based loops
  freely homotopic to $\alpha$ and $\beta$ respectively,
  then $i_{G}(\alpha, \beta) > 0$.
 \end{thm}

 \begin{rmk}\label{rmk:uselessness}
  \begin{enumerate}[(1)]
   \item When both of the curves are separating on $\Sigma$,
	 the theorem is trivial 
	 since
	 $\ell(a)(|b|)+\ell(b)(|a|)=0 \in \mathbb{Z}|a|+\mathbb{Z}|b|$.
   \item In the case where both of the curves are non-separating on $\Sigma$,
	 if $|a|$ and $|b|$ are linearly dependent on $H$,
	 the theorem is also trivial, 
	 namely
	 it is always true that $\ell(a)(|b|)+\ell(b)(|a|) \in \mathbb{Z}|a|+\mathbb{Z}|b|$.
	 See Theorem \ref{thm:non-sep-2} in \S \ref{sec:proof}.
  \end{enumerate}
 \end{rmk}

\subsection{Example of computations in Theorem \ref{thm:main}}

  Let $\alpha$ and $\beta$ be oriented simple closed curves on $\Sigma_{2,1}$ as shown in
  the left part of Figure \ref{fig:example}.
  Taking symplectic generators of $\pi$ as in the right of the figure,
  we may choose elements $a=x_{1}x_{2}y_{2}x_{2}^{-1}$ and $b=y_{2}x_{1}^{-1}$ in $\pi$,
  which are based homotopy classes of based loops freely homotopic to $\alpha$ and $\beta$ respectively.
  %
  %
  We can confirm that the homology classes $|a|=X_{1}+Y_{2}$ and $|b|=-X_{1}+Y_{2}$ meet $|a|{\cdot}|b| = 0$.
  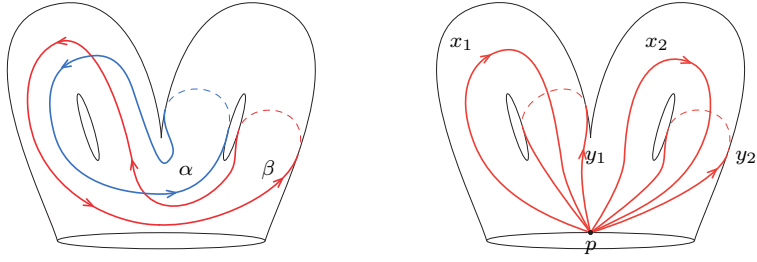
\begin{figure}[h]
   \begin{center}\unitlength=1mm
    \begin{picture}(100,35)(0,-2)
     \footnotesize
     \put(0,-3){\includegraphics[height=33mm]{%
     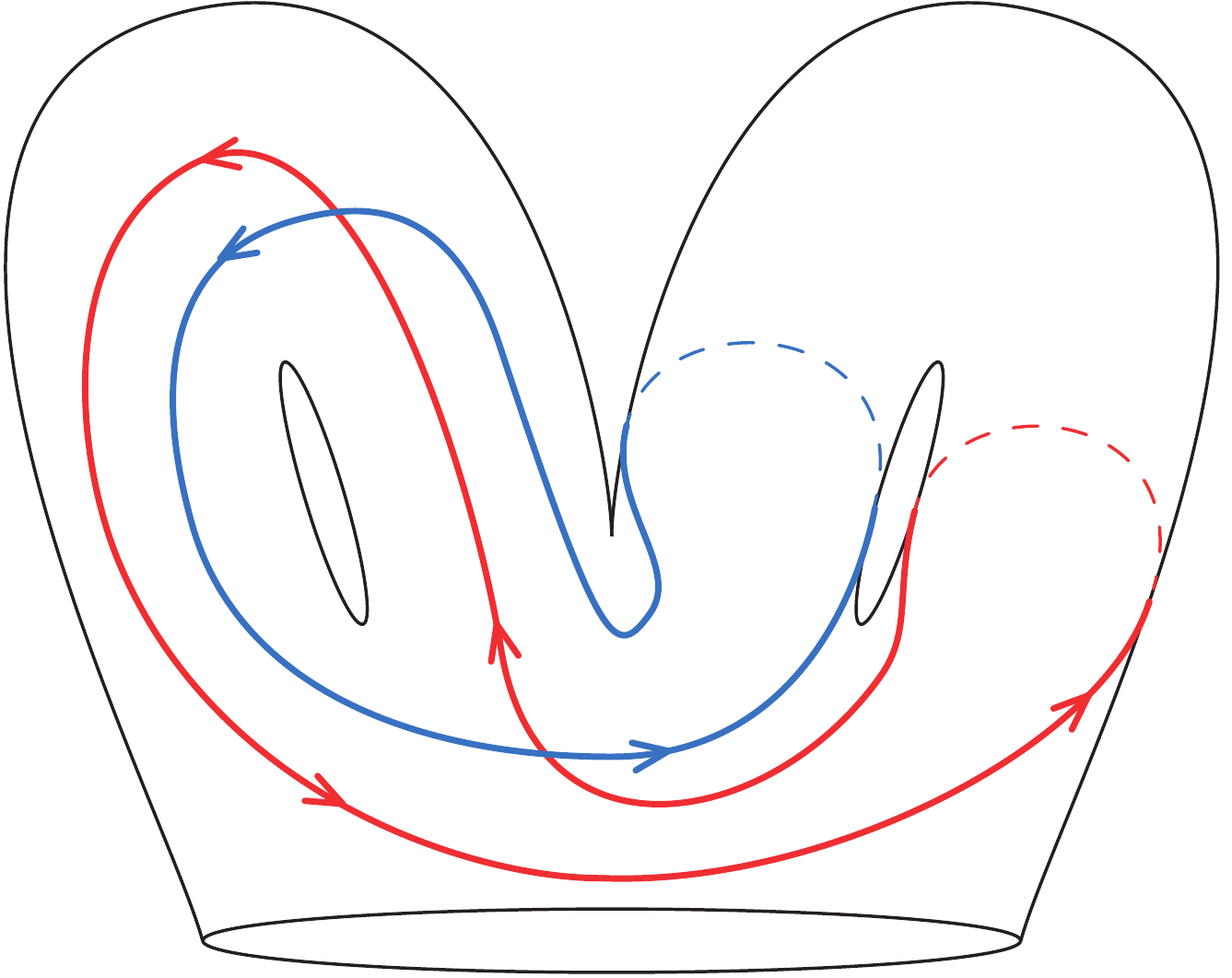}}
     \put(57,-3){\includegraphics[height=33mm]{%
     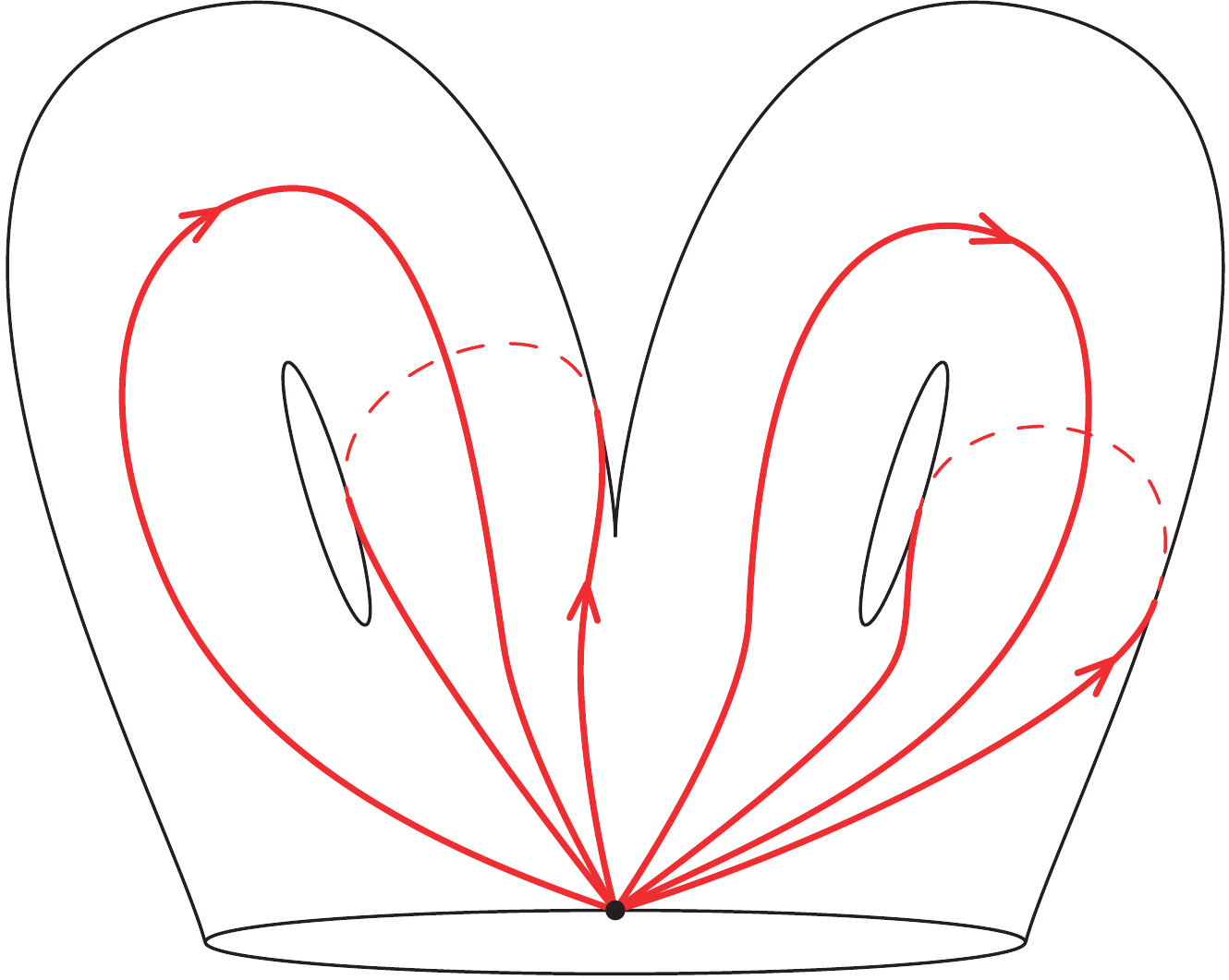}}
     \put(23,7){$\alpha$}\put(34,7){$\beta$}
     \put(59,24){$x_{1}$}\put(77,9){$y_{1}$}
     \put(85,24){$x_{2}$}\put(97,9){$y_{2}$}
     \put(77,-3){$p$}
    \end{picture}
   \end{center}
   \caption{Oriented curves on $\Sigma_{2,1}$ and symplectic generators of $\pi$.}
   \label{fig:example}
  \end{figure}

 %
 %
 We now compute $\ell(a)$ and $\ell(b)$.
 %
 In addition to
 the property of the map $\ell$ mentioned above,
 we will use the following.
 \begin{equation*}
  \ell(ghg^{-1})=\ell(h)+|g|\wedge|h|,~ \forall g,h \in \pi.
 \end{equation*}
 (See Proposition \ref{prop:ell2} for a proof).
 \begin{eqnarray*}
  \ell(a)
   &=& \ell(x_{1})+\ell(x_{2}y_{2}x_{2}^{-1})
   +\frac{1}{2}X_{1}\wedge Y_{2} \\
  &=& \frac{1}{2}
   \left(
    X_{1}\wedge Y_{1}
    +X_{2}\wedge Y_{2}
    +X_{1}\wedge Y_{2}
   \right), \\
  \ell(b)
  &=& \ell(y_{2})-\ell(x_{1})-\frac{1}{2}Y_{2}\wedge X_{1} \\
  &=& \frac{1}{2}
   \left(
    -X_{2}\wedge Y_{2}
    -X_{1}\wedge Y_{1}
    -Y_{2}\wedge X_{1}
   \right). 
 \end{eqnarray*}
 Elements $\ell(a)$ and $\ell(b)$ in $H \wedge H$
 act on $|b|$ and $|a|$ respectively as follows;
 \begin{eqnarray*}
  \ell(a)|b|
   &=& \frac{1}{2}
   \left(
    X_{1}\wedge Y_{1}
    +X_{2}\wedge Y_{2}
    +X_{1}\wedge Y_{2}
   \right)(-X_{1}+Y_{2})
   ~=~
   \frac{1}{2}(X_{1}-Y_{2}), \\
  \ell(b)|a|
  &=& \frac{1}{2}
  \left(
   -X_{2}\wedge Y_{2}
   -X_{1}\wedge Y_{1}
   -Y_{2}\wedge X_{1}
  \right)(X_{1}+Y_{2})
  ~=~ \frac{1}{2}(X_{1}+Y_{2}).
 \end{eqnarray*}
 Now we have that
\begin{equation*}
 \ell(a)|b|+\ell(b)|a|
  = X_{1}
  = \frac{|a|-|b|}{2} \not\in \mathbb{Z}|a|+\mathbb{Z}|b|,
\end{equation*}
and see from Theorem \ref{thm:main} that $i_{G}(\alpha, \beta) > 0$.

\begin{rmk}
 The converse of Theorem \ref{thm:main} does not hold.
 The following is a counter-example:
 We set $\Sigma = \Sigma_{2,1}$.
 The symplectic generators of $\pi$ as above.
 Let $\alpha$ be an oriented simple closed curve on $\Sigma$ freely homotopic to
 $a=x_{1}$.
 Take a word $b'=x_{2}^{-1} \in \pi$.
 It is easy to see that
 $\ell(a)|b'|+\ell(b')|a| = 0$.
 We now consider a word $b=b'[y_{1},\zeta]\zeta$,
 where $\zeta=[x_{1},y_{1}][x_{2},y_{2}]$, the homotopy class of the boundary loop.
 We can confirm that $b$
 has a representative of a simple loop, say $\beta$,
 and $i_{G}(\alpha, \beta) = 2$ by the bigon criterion.
 Setting $\omega := \ell(\zeta) = X_{1}\wedge Y_{1} + X_{2}\wedge Y_{2}$,
 we have that $\ell(b)=\ell(b')+\omega$
 (see Proposition \ref{prop:ell2}).
 From facts that $|b|=|b'|$ and $\omega(|a|)=-|a| \in \mathbb{Z}|a|$ (see Proposition \ref{prop:omega|a|}),
 we obtain
 \begin{equation*}
  \ell(a)|b|+\ell(b)|a|
   = \ell(a)|b'|+\ell(b')|a|+\omega(|a|)
   = -|a| \in \mathbb{Z}|a| \subset \mathbb{Z}|a|+\mathbb{Z}|b|.
 \end{equation*}
\end{rmk}

\section*{Acknowledgments}\label{sec:ack}
The author would like to thank Yasushi Kasahara and Nariya Kawazumi
for valuable comments.

\section{Preliminaries}\label{sec:prelim}

In this section we recall some of notions in the theory about the mapping class group of a
surface developed in \cite{Ka}, \cite{M} and \cite{KK},
and prepare propositions, most of all are with a proof, for use in \S \ref{sec:proof}.

We denote by $\widehat{T}$ the completed tensor algebra generated by $H$,
i.e., $\widehat{T} = \prod_{m=0}^{\infty}H^{\otimes m}$.
We will write the tensor product of $u$ and $v$ in $\widehat{T}$
as $uv$ omitting the tensor symbol.
We set $\widehat{T}_{p} = \prod_{m \geq p}^{\infty}H^{\otimes m}$ for $p \geq 1$,
which gives a filtration of $\widehat{T}$.


\subsection{Magnus expansion of $\pi$}
The subset $1+\widehat{T}_{1}$ of $\widehat{T}$ has a group structure by the tensor product of $\widehat{T}$.
\begin{dfn}[Kawazumi, \cite{Ka}]
 A Magnus expansion $\theta$ of $\pi$ is a group homomorphism from $\pi$ to $1+\widehat{T}_{1}$ satisfying
 \begin{equation*}
  \theta(x) \equiv 1+|x| ~(\mod{\widehat{T}_{2}}) ~\text{for } \forall x \in \pi.
 \end{equation*}
\end{dfn}
For a Magnus expansion $\theta$, we set $\ell^{\theta} = \log \circ\, \theta$,
where the logarithm map is defined as
\begin{equation*}
 \log(u) := \sum_{k=1}^{\infty} \frac{(-1)^{k-1}}{k}(u-1)^{k} ~\text{for }\forall u \in 1+\widehat{T}_{1}.
\end{equation*}
We denote by $\widehat{\mathcal{L}}$
the completed free Lie algebra generated by $H$ with the bracket $[u,v] = uv-vu$
and set $\mathcal{L}_{p}=\widehat{\mathcal{L}} \cap H^{\otimes p}$ for $p \geq 1$.
%
\begin{dfn}
 A group-like expansion of $\pi$ is a Magnus expansion $\theta$ of $\pi$ satisfying
 \begin{equation*}
  \ell^{\theta}(\pi) \subset \widehat{\mathcal{L}}.
 \end{equation*}
\end{dfn}

We take symplectic generators $\{x_{j}, y_{j}\}_{1 \leq j \leq g}$ of $\pi$
as shown in Figure \ref{fig:symp_gene}.
Let $\zeta$ be a homotopy class of the boundary loop of $\Sigma$
as in the same figure,
which is expressed in the word of the symplectic generators as
$\zeta = \prod_{j=1}^{g}[x_{j},y_{j}]$.
Set $\omega=\sum_{j=1}^{g} [X_{j}, Y_{j}] \in \mathcal{L}_{2}$,
where $X_{j} = |x_{j}|$ and $Y_{j} = |y_{j}|$ ($j=1,2,\dots,g$).
%
\begin{dfn}[Massuyeau, \cite{M}]
 A symplectic expansion of $\pi$ is a group-like expansion $\theta$ of $\pi$ satisfying
 \begin{equation*}
  \ell^{\theta}(\zeta) = \omega.
 \end{equation*}
\end{dfn}

 \subsection{The logarithm of Dehn twists}
 Let $\operatorname{Aut}(\widehat{T})$ be the group of the filter-preserving algebra automorphisms
 of $\widehat{T}$,
 and $\operatorname{Der}(\widehat{T})$ the space of the derivations of $\widehat{T}$.
 We may identify $\widehat{T}_{1}$ with $\operatorname{Der}(\widehat{T})$ as follows:

 By the Poincar\'{e} duality $\widehat{T}_{1}=H\otimes \widehat{T}$ is identified with
 $\operatorname{Hom}(H, \widehat{T})$,
 namely, for $X_{1}X_{2}\cdots X_{k} \in H^{\otimes k}$ and $Y \in H$ we define
 \begin{equation}\label{111614_8Sep16}
  (X_{1}X_{2}\cdots X_{k})(Y) = (Y{\cdot}X_{1})X_{2}\cdots X_{k}.
 \end{equation}
 Since each $h \in \operatorname{Hom}(H, \widehat{T})$ determines
 a derivation $D_{h} \in \operatorname{Der}(\widehat{T})$ by
 \begin{equation*}
  D_{h}(Y_{1}Y_{2}\cdots Y_{p})
   = h(Y_{1})Y_{2}\cdots Y_{p}+Y_{1}h(Y_{2})\cdots Y_{p}+\cdots+Y_{1}Y_{2}\cdots h(Y_{k})
 \end{equation*}
 and vise versa,
 we have an identification of $\widehat{T}_{1}$ with $\operatorname{Der}(\widehat{T})$.


We denote by $\operatorname{Aut}(\pi)$ the automorphism group of $\pi$.
%
  For $\phi \in \operatorname{Aut}(\pi)$
 An automorphism $T^{\theta}(\phi) \in \operatorname{Aut}(\widehat{T})$
 is called the total Johnson map of $\phi$
 associated to a Magnus expansion $\theta$,
 if it satisfies
\begin{equation*}
 T^{\theta}(\phi)\circ\theta = \theta\circ\phi.
\end{equation*}
 It is known \cite[\S 2.5]{Ka} that
 there uniquely exists a total Johnson map for each $\phi \in \operatorname{Aut}(\pi)$.

 We denote by $t_{\alpha}$ a right-handed Dehn twist along a simple closed curve $\alpha$ on $\Sigma$.
 In \cite{KK}, N.\ Kawazumi and Y.\ Kuno gave an explicit description of the total Johnson map
 of $t_{\alpha}$ with respect to a symplectic expansion $\theta$ as follows:

 Let $\nu$ be a linear map on $\widehat{T}$, called the cyclic permutation,
 defined by
 \begin{equation*}
  \nu(X_{1}X_{2}\cdots X_{k}) = X_{2}\cdots X_{k}X_{1} ~~(X_{j} \in H, j=1,2,\dots,k),
 \end{equation*}
 and $N$ a linear map on $\widehat{T}$ defined by
 \begin{equation*}
  N|_{H^{\otimes k}} = \sum_{j=0}^{k-1} \nu^{j},
 \end{equation*}
 for $k \geq 1$, and $N|_{\mathbb{Q}} = 0$.
%
%
  Kawazumi and Kuno \cite[\S 2.6]{KK} introduced a map $L^{\theta}: \pi \rightarrow \widehat{T}_{2}$
  with respect to a Magnus expansion $\theta$
  defined by
  \begin{equation*}
   L^{\theta}(x) = \frac{1}{2}N(\ell^{\theta}(x)\ell^{\theta}(x)),~ [x \in \pi],
  \end{equation*}
  and they showed that $L^{\theta}(x^{-1}) = L^{\theta}(x)$ and
  $L^{\theta}(yxy^{-1}) = L^{\theta}(x)$ for any $x,y \in \pi$.

  Let $\alpha$ be a simple closed curve on $\Sigma$.
  The value $L^{\theta}(a) \in \widehat{T}_{2}$ is well-defined
  because of the property of $L^{\theta}$ mentioned above.
  With respect to a symplectic expansion $\theta$,
  an explicit description of the total Johnson map of $t_{a}$
  in the view of the identification
  $\widehat{T}_{1}=\operatorname{Der}(\widehat{T})$
  is given \cite[Theorem 1.1.1]{KK} as
  \begin{equation}\label{130101_27Sep16}
   T^{\theta}(t_{a}) = e^{-L^{\theta}(a)},
  \end{equation}
  namely, for $\forall x \in \pi$
  \begin{equation*}
  \theta(t_{a}(x)) = e^{-L^{\theta}(a)}(\theta(x)),
  \end{equation*}
  where the exponential map on $\widehat{T}$ is defined as
  $e^{u}:=\sum_{k=0}^{\infty} \frac{u^{k}}{k!}$ for $u \in \widehat{T}$.


  \subsection{Degree $2$-part and $3$-part of $L^{\theta}(x)$}
 In $H^{\otimes 2}$, we will identify $\mathcal{L}_{2}$ with $H \wedge H$
 by identifying $[X,Y] = XY-YX$ with $X \wedge Y$ for $X,Y \in H$.
 In $H^{\otimes 3}$, we will regard $\wedge^{3}H$ as a subspace of $H^{\otimes 3}$
 by identifying $X{\wedge}Y{\wedge}Z$ with $XYZ+YZX+ZXY-XZY-ZYX-YXZ$.
 Thus an action of $H \wedge H$ and of $\wedge^{3}H$ on $H$
 are given as follows.
  \begin{prop}\label{prop:actions}\normalfont
   For $X, Y, Z \in H$ and $u \in H \wedge H = \mathcal{L}_{2}$,
   \begin{enumerate}[(1)]
    \item  $(X \wedge Y)(Z) = (Z{\cdot}X)Y - (Z{\cdot}Y)X$,
    \item  $(X \wedge u)(Z) = (Z{\cdot}X)u - X \wedge (u(Z))$.
   \end{enumerate}
  \end{prop}
  \begin{proof}
   \begin{enumerate}[(1)]
    \item Computing the action of $X\wedge Y=[X,Y]=XY-YX$ on $H$ as (\ref{111614_8Sep16}),
	  we have
	  \begin{equation*}
	  (X\wedge Y)(Z)
	  = (XY)(Z) - (YX)(Z)
	  = (Z{\cdot}X)Y - (Z{\cdot}Y)X.
	  \end{equation*}
    \item It is enough to show the equation for $u=A \wedge B$, [$A, B \in H$].
	  %
	  \begin{eqnarray*}
	   (X\wedge (A \wedge B))(Z)
	    &=& (XAB+ABX+BXA-XBA-BAX-AXB)(Z) \\
	   &=& (Z{\cdot}X)A\wedge B-X \wedge \left\{
					      (Z{\cdot}A)B +(Z{\cdot}B)A
					     \right\} \\
	   &=& (Z{\cdot}X)A\wedge B-X \wedge \left\{
					      (A \wedge B)(Z)
					     \right\}. 
	  \end{eqnarray*}
   \end{enumerate}
  \end{proof}
  For a Magnus expansion $\theta$,
  we denote the degree $k$-part of $\ell^{\theta}(x)$ and $L^{\theta}(x)$
  by $\ell_{k}^{\theta}(x)$ and $L_{k}^{\theta}(x)$ respectively ($x \in \pi$).
  %
  %
  \begin{prop}\label{prop:derivations}\normalfont
   Let $\theta$ be a Magnus expansion.
   For $\forall a \in \pi$ and $\forall u \in \mathcal{L}_{2}$,
   \begin{enumerate}[(1)]
    \item $L^{\theta}_{2}(a)(u)=-|a| \wedge (u(|a|))$,
    \item $L^{\theta}_{2}(a)^{2}(u)=0$.
   \end{enumerate}
  \end{prop}
  \begin{proof}
   For $a \in \pi$,
   $L_{2}^{\theta}(a) = \ell^{\theta}_{1}(a)\ell^{\theta}_{1}(a) = |a||a|$.
   \begin{enumerate}[(1)]
    \item 
	  For $X{\wedge}Y \in H{\wedge}H$ we have
	 \begin{eqnarray*}
	  L_{2}(a)(X \wedge Y)
	   &=& |a||a|(XY-YX) \\
	  &=& (X{\cdot}|a|)|a|\wedge Y-(Y{\cdot}|a|)|a|\wedge X \\
	  &=& |a|\wedge \left\{ (X\wedge Y)(|a|) \right\}.
	 \end{eqnarray*}
	  %
	  %
    \item For $\forall u \in \mathcal{L}_{2}$,
	  putting $Y := u(|a|)$, we have
	  \begin{equation*}
	   L^{\theta}_{2}(a)^{2}(u)
	    = |a||a|\left( |a|\wedge Y \right)
	    = 0+|a|(Y{\cdot}|a|)|a| - (Y{\cdot}|a|)|a||a| - 0
	    = 0.
	  \end{equation*}
   \end{enumerate}
  \end{proof}

  When $\theta$ is an group-like expansion,
   the map $\ell^{\theta}_{2}$ sends $\forall x \in\pi$ into $\mathcal{L}_{2} = H\wedge H$.
  %
  \begin{prop}[\cite{KK}, Lemma 6.4.1.]\label{prop:L_{3}}
   For a group-like expansion $\theta$ and for $a \in \pi$,
  \begin{equation*}
   L_{3}^{\theta}(a) = |a|\wedge \ell^{\theta}_{2}(a) \in \wedge^{3}H.
  \end{equation*}
  \end{prop}
  %
%
 From the proposition above and Proposition \ref{prop:actions}
 we can compute the value $L_{3}^{\theta}(a)(X)$ for $X \in H$ as
   \begin{equation*}\label{133254_20Jul16}
    L_{3}^{\theta}(a)(X)
   =(X{\cdot}|a|)\ell^{\theta}_{2}(a)
   -|c|{\wedge}\left\{
		\ell^{\theta}_{2}(a)(X)
	       \right\}.
   \end{equation*}
%
For $\phi \in \operatorname{Aut}(\pi)$,
the map defined by
	\begin{equation*}
	 \tau^{\theta}(\phi) := T^{\theta}(\phi) \circ |\phi|^{-1}
	\end{equation*}
 is called the Johnson map of $\phi$ with respect to a Magnus expansion $\theta$.
 We denote by $\tau_{k}^{\theta}(\phi)(\gamma)$
 the $(k+1)$-th degree part of $\tau^{\theta}(\phi)(\gamma)$ for $\gamma \in \pi$,
 namely
	\[
	\tau^{\theta}(\phi)|_{H}
	= \operatorname{id}_{H}+\sum_{k=1}^{\infty}\tau^{\theta}_{k}(\phi).
	\]
	%

The following proposition proven in \cite{KK} is deduced from (\ref{130101_27Sep16}).
 \begin{prop}[\cite{KK},Theorem 6.5.3.]\label{prop:nonsepa-tau}
	 For a non-separating simple closed curve $\alpha$ on $\Sigma$
	 and a symplectic expansion $\theta$,
  \begin{equation*}
   \tau^{\theta}_{1}(t_{\alpha}) = -L_{3}^{\theta}(\alpha).
  \end{equation*}
 \end{prop}
	 %

\subsection{The map $\ell^{\theta}_{2}$}

Let $\theta$ be a Magnus expansion.
We denote by $\theta_{2}(x)$ the degree $2$ part of $\theta(x)$ for $x \in \pi$.
	%
	\begin{prop}\label{prop:theta2_ell2}
	 For $\forall a \in \pi$,
	 $\ell^{\theta}_{2}(a)
	 =\theta_{2}(a)-\frac{1}{2}|a|^{2}$.
	\end{prop}
	\begin{proof}
	 Module $\widehat{T}_{3}$,
	 \begin{equation*}
	  \log(\theta(a)) \equiv \log(1+|a|+\theta_{2}(a))
	   =\sum_{k=1}^{\infty}\frac{(-1)^{k-1}}{k}(|a|+\theta_{2}(a))^{k}
	   \equiv |a|+\theta_{2}(a)-\frac{1}{2}|a|^{2}.
	 \end{equation*}
	\end{proof}
	It is easy to see that $\ell^{\theta}_{2}=\theta_{2}$ on $[\pi, \pi]$,
	especially $\ell^{\theta}_{2}(1) = \theta_{2}(1) = 0$.
	\begin{prop}\label{prop:ell2}\normalfont
	 For $\forall a, b \in \pi$,
	 \begin{enumerate}[(1)]
	  \item $\ell^{\theta}_{2}(ab)
		=\ell^{\theta}_{2}(a)+\ell^{\theta}_{2}(b)
		+\frac{1}{2}|a|\wedge |b|$,
	  \item $\ell^{\theta}_{2}(a^{-1})=-\ell^{\theta}_{2}(a)$,
	  \item $\ell^{\theta}_{2}(aba^{-1})=\ell(b)+|a|\wedge|b|$,
	  \item $\ell^{\theta}_{2}([a,b])=|a|\wedge|b|$.
	 \end{enumerate}
	\end{prop}
	\begin{proof}
	 \begin{enumerate}[(1)]
	  \item Module $\widehat{T}_{3}$,
		\begin{equation*}
		 \theta(ab) \equiv (1+|a|+\theta_{2}(x))(1+|b|+\theta_{2}(y))
		  \equiv 1+|a|+|b|+|a||b|+\theta_{2}(a)+\theta_{2}(b).
		\end{equation*}
		Thus $\theta_{2}(ab)=|a||b|+\theta_{2}(a)+\theta_{2}(b)$.
		By using Proposition \ref{prop:theta2_ell2},
		\begin{eqnarray*}
		 \ell^{\theta}_{2}(ab)
		 &=& |a||b|+\theta_{2}(a)+\theta_{2}(b) - \frac{1}{2}(|a|^{2}+|a||b|+|b||a|+|b|^{2}) \\
		 &=& \ell^{\theta}_{2}(a)
		  + \ell^{\theta}_{2}(b)
		  + \frac{1}{2}(|a||b|-|b||a|). 
		\end{eqnarray*}
	  \item Using (1), we have $0=\ell^{\theta}_{2}(aa^{-1})
		=\ell^{\theta}_{2}(a)+\ell^{\theta}_{2}(a^{-1})$.
	  \item Using (1) and (2),
		\begin{eqnarray*}
		 \ell^{\theta}_{2}(aba^{-1})
		 &=& \ell^{\theta}_{2}(a)
		  +\ell^{\theta}_{2}(b) + \ell^{\theta}_{2}(a^{-1}) + \frac{1}{2}|b|\wedge|a^{-1}|
		  +\frac{1}{2}|a|\wedge|b| \\
		 &=& \ell(b)+|a|\wedge|b|.
		\end{eqnarray*}
	  \item Using (1) and (3),
		\begin{equation*}
		 \ell^{\theta}_{2}([a,b])
		  = \ell^{\theta}_{2}(aba^{-1})+\ell^{\theta}_{2}(b^{-1})
		  = \ell^{\theta}_{2}(b)+|a|\wedge|b|-\ell^{\theta}_{2}(b)
		  = |a|\wedge|b|.
		\end{equation*}
		%
		%
	 \end{enumerate}
	 %
	 %
	\end{proof}
        Note that we can define the map $\ell^{\theta}_{2}$ from
        Proposition \ref{prop:ell2} (1)
        and the initial data,
        $\ell^{\theta}_{2}(1)=0$
        and its values on generators of $\pi$.
        The map $\ell$ mentioned in \S \ref{subsec:main_theorem}
        is nothing but the map $\ell^{\theta}_{2}$
        with respect to a specific symplectic expansion $\theta$ (See \S \ref{subsec:ell}).

	Now let $\theta$ be a symplectic expansion.
	\begin{prop}\label{prop:Q|c|}
	 If $c \in \pi$ has a representative which is freely homotopic to
	 a non-separating simple closed curve on $\Sigma$,
	 then $\ell_{2}^{\theta}(c)(|c|) \in \mathbb{Q}|c| \subset H$.
	\end{prop}
	\begin{proof}
	 It has shown in \cite[Proposition 6.5j.1]{KK} that $L^{\theta}_{2}(c)L^{\theta}_{3}(c) = 0$
	 on $H$.
	 On the other hand,
	 for any $X \in H$, putting $Y := \ell^{\theta}_{2}(c)(X)$, we can compute
	 \begin{eqnarray*}
	  L^{\theta}_{2}(c)L^{\theta}_{3}(c)(X)
	  &=&
	  L^{\theta}_{2}(c)\left(
			    (X{\cdot}|c|)\ell^{\theta}_{2}(c) - |c|\wedge Y
			   \right) \\
	  &=&
	   -(X{\cdot}|c|)|c|\wedge \left\{\ell^{\theta}_{2}(c)(|c|)\right\}
	   + |c|\wedge \left\{
			(|c|\wedge Y)(|c|)
		       \right\} \\
	  &=&
	   -(X{\cdot}|c|)|c|\wedge \left\{\ell^{\theta}_{2}(c)(|c|)\right\}.
	 \end{eqnarray*}
	 %
	 Hence we obtain
	 $|c| \wedge \left\{\ell^{\theta}_{2}(c)(|c|)\right\}=0$.
	 This proves the proposition.
	\end{proof}
	 The action of $\omega=\sum_{i=1}^{g}X_{i}\wedge Y_{i} \in H \wedge H$ on $H$
	 is as follows:
	 For $A = \sum_{j=1}^{g} \xi_{j}X_{j}+\eta_{j}Y_{j} \in H$,
	 ($\xi_{j}, \eta_{j} \in \mathbb{Q}$),
	 \begin{equation*}
	  {\omega}(A)
	   = \sum_{i=1}^{g} (A{\cdot}X_{i})Y_{i}-(A{\cdot}Y_{i})X_{i}
	   = \sum_{i=1}^{g} -\eta_{i}Y_{i}-\xi_{i}X_{i}
	   = -A.
	 \end{equation*}
	 Since $\theta$ is a symplectic expansion, i.e.,
	 $\ell^{\theta}_{2}(\zeta) = \omega$,
	 we have the following:
  	\begin{prop}\label{prop:omega|a|}
	 For the boundary curve $\zeta$ of $\Sigma_{g,1}$
	 and $\forall a \in \pi$,
	 $\ell^{\theta}_{2}(\zeta)(|a|)=-|a|$.
	\end{prop}
	%
	%


\section{Proof of Theorem \ref{thm:main}}\label{sec:proof}

 Let $\alpha, \beta$ be oriented simple closed curves on $\Sigma$.
 For the reason mentioned in Remark \ref{rmk:uselessness}(1),
 we may assume that $\alpha$ is non-separating on $\Sigma$.
 %
 We choose based loops $a$ and $b$ with the base point $p \in \partial\Sigma$
 which are freely homotopic to $\alpha$ and $\beta$ respectively,
  and we will denote the based homotopy classes of them by the same symbols $a$, $b$.


 Here we recall the classical formula about the action of $t_{\alpha}$, a right-handed Dehn twist along $\alpha$,
 on $H$;
 \begin{equation}\label{123029_27Sep16}
  |t_{\alpha}|(X) = X + (|a|{\cdot}X)|a|,~ (\forall X \in H).
 \end{equation}

\subsection{Degree two part of symplectic expansion of $t_{\alpha}(b)$}
 Focusing on the $H^{\otimes 2}$-part of symplectic expansion of
 $t_{\alpha}(b)$ and of $b$
 we have the following.
 \begin{lem}\label{lem:theta_2}
  Curves $\alpha$, $\beta$ and  $a, b \in \pi$ as above.
  Let $\theta$ be a symplectic expansion of $\pi$.
  Suppose that $i_{A}(\alpha, \beta)=0$.
  Then
  \begin{equation*}
   \theta_{2}(t_{\alpha}(b))-\theta_{2}(b) = |a|\wedge(\ell^{\theta}_{2}(a)|b|+\ell^{\theta}_{2}(b)|a|).
  \end{equation*}
 \end{lem}
 \begin{proof}
  The expansion of $t_{\alpha}(b)$ by a symplectic expansion $\theta$ of $\pi$
  modulo $\widehat{T}_{3}$ is computed as
  \begin{eqnarray}
   \theta(t_{\alpha}(b))
    &=& T^{\theta}(t_{\alpha})(\theta(b)) \nonumber \\
   &\equiv& T^{\theta}(t_{\alpha})(1+|b|+\theta_{2}(b))
    ~~\mod{\widehat{T}_{3}} \nonumber \\
   &=& 1+ T^{\theta}(t_{\alpha})(|b|)
    +T^{\theta}(t_{\alpha})(\theta_{2}(b)). \label{173635_29Sep16} 
  \end{eqnarray}
  %
  \begin{clm}
   $T^{\theta}(t_{\alpha})(|b|)
   \equiv |b| + |a| \wedge \left\{\ell^{\theta}_{2}(a)(|b|)\right\}
   ~\mod{\widehat{T}_{3}}$
  \end{clm}
  \begin{proof}[Proof of Claim 1.]
   It follows
   from the classical formula (\ref{123029_27Sep16})
   and the condition $|a|{\cdot}|b|=0$
   that %
   \begin{equation*}
    |t_{\alpha}(b)| = |b|.
   \end{equation*}
   Recall the definition of the Johnson map:
   \begin{equation*}
    T^{\theta}(t_{\alpha}) = \tau^{\theta}(t_{\alpha}) \circ |t_{\alpha}|.
   \end{equation*}
   Then we have
   \begin{equation}\label{113028_27Sep16}
    T^{\theta}(t_{\alpha})(|b|)
    = \tau^{\theta}(t_{\alpha})(|b|)
     \equiv |b|+\tau_{1}^{\theta}(t_{\alpha})(|b|)
     ~~\mod{\widehat{T}_{3}}.
   \end{equation}
   %
   Applying Proposition \ref{prop:nonsepa-tau} (note that we assume $\alpha$ is non-separating loop),
   \ref{prop:L_{3}} and \ref{prop:actions}(2),
   \begin{equation*}
    \tau_{1}^{\theta}(t_{\alpha})(|b|)
     = -L_{3}^{\theta}(\alpha)(|b|)
     = -(|a|\wedge\ell^{\theta}_{2}(a))(|b|)
     = |a|\wedge\left\{ \ell^{\theta}_{2}(a)(|b|) \right\}.
   \end{equation*}
  \end{proof}

  \begin{clm}
   $T^{\theta}(t_{\alpha})(\theta_{2}(b))
   \equiv \theta_{2}(b) + |a| \wedge \left\{ \ell^{\theta}_{2}(b)(|a|) \right\}
   ~\mod{\widehat{T}_{3}}$
  \end{clm}
  \begin{proof}[Proof of Claim 2.]
   Note that the map $|t_{\alpha}| \in \operatorname{Hom}(H, \widehat{T}_{1})$
   is also regarded as an automorphism of $\widehat{T}$ such that
   $|t_{\alpha}|(X_{1}\cdots X_{p})
   =|t_{\alpha}|^{\otimes p}(X_{1}\cdots X_{p})
   = |t_{\alpha}|(X_{1})\cdots|t_{\alpha}|(X_{1})$
   for $X_{1}\cdots X_{p} \in H^{\otimes p}$.
   Thus we have
   \begin{equation*}
    T^{\theta}(t_{\alpha})(\theta_{2}(b))
     = \tau^{\theta}(t_{\alpha})(|t_{\alpha}|^{\otimes 2}\theta_{2}(b))
   \end{equation*}
   %
   The lowest degree part of
   $\tau^{\theta}(t_{\alpha})(u)$ for $u \in H^{\otimes 2}$
   is equal to $u$
   since that of $\tau^{\theta}(t_{\alpha})(X)$ for $X \in H$ is $X$
   as we can see in (\ref{113028_27Sep16}).
   Applying Proposition \ref{prop:theta2_ell2},
  \begin{eqnarray*}
   \tau^{\theta}(t_{\alpha})(|t_{\alpha}|^{\otimes 2}\theta_{2}(b))
    &\equiv& |t_{\alpha}|^{\otimes 2}\theta_{2}(b)
    ~~ \mod{\widehat{T}_{3}} \\
   &=& |t_{\alpha}|^{\otimes 2}\ell^{\theta}_{2}(b)
    +\frac{1}{2}|t_{\alpha}|^{\otimes 2}|b|^{2}. 
  \end{eqnarray*}
   Now we claim that $\forall u \in \mathcal{L}_{2}$,
   $|t_{\alpha}|^{\otimes 2}(u) = u + |a| \wedge \left\{ u(|a|) \right\}$.
   In fact, for $X \wedge Y \in \mathcal{L}_{2}$,
   \begin{eqnarray*}
    |t_{\alpha}|^{\otimes 2} (X\wedge Y)
    &=& (X+(|a|{\cdot}X)|a|) \wedge (Y+(|a|{\cdot}Y)|a|) \\
    &=& X \wedge Y + |a| \wedge \left\{ (X \wedge Y)(|a|) \right\}. 
   \end{eqnarray*}
   We then have
   \begin{equation*}
    \tau^{\theta}(t_{\alpha})(|t_{\alpha}|^{\otimes 2}\theta_{2}(b))
     = \ell^{\theta}_{2}(b)+|a|\wedge\left\{ \ell^{\theta}_{2}(b)(|a|) \right\}
     +\frac{1}{2}|b|^{2}
     = \theta_{2}(b)+|a|\wedge\left\{ \ell^{\theta}_{2}(b)(|a|) \right\}.
   \end{equation*}
  \end{proof}

  From (\ref{173635_29Sep16}) with Claim 1 and Claim 2 we obtain
  \begin{equation*}
  \theta(t_{\alpha}(b))
	\equiv
	1+|b|
	+\theta_{2}(b)
	+|a|\wedge\left\{
	\ell^{\theta}_{2}(a)(|b|)+\ell^{\theta}_{2}(b)(|a|)
	\right\}
	~~\mod{\widehat{T}_{3}},
  \end{equation*}
  and a proof of the lemma has done.
 \end{proof}

 \subsection{The value $\ell^{\theta}_{2}(a)(|b|)+\ell^{\theta}_{2}(b)(|a|)$}
When $i_{G}(\alpha, \beta) = 0$ holds,
there exists a based homotopy class $b_{0} \in \pi$
corresponding to the curve $\beta$ so that
the Dehn twist $t_{\alpha}$ fixes it.
Then $\theta(t_{\alpha}(b_{0})) = \theta(b_{0})$,
especially $\theta_{2}(t_{\alpha}(b_{0})) - \theta_{2}(b_{0}) = 0$.
According to the lemma above,
if we take a symplectic expansion $\theta$,
we can restate the last equation that
$|a| \wedge \left\{
\ell^{\theta}_{2}(a)(|b_{0}|)+\ell^{\theta}_{2}(b_{0})(|a|)
\right\} = 0$.

In this subsection,
we discuss the behavior of the value $\ell^{\theta}_{2}(a)(|b|)+\ell^{\theta}_{2}(b)(|a|)$
for arbitrary chosen homotopy classes $|a|$ and $|b|$ of the curves.
In the case where the two curves $\alpha$ and $\beta$ are both non-separating on $\Sigma$,
we first suppose that corresponding homotopy classes $|a|$ and $|b|$ of the curves
are linearly independent on $H$,
and show that
the value $\ell^{\theta}_{2}(a)(|b|)+\ell^{\theta}_{2}(b)(|a|)$
with any symplectic expansion $\theta$
is in $\mathbb{Z}|a| + \mathbb{Z}|b| \subset H$ if $i_{G}(\alpha, \beta)=0$.
Second,
we see that,
when $|a|$ and $|b|$ are linearly dependent on $H$,
$\ell^{\theta}_{2}(a)(|b|)+\ell^{\theta}_{2}(b)(|a|)$
with a certain type of symplectic expansion $\theta$
is always in $\mathbb{Z}|a| + \mathbb{Z}|b|$ ($=\mathbb{Z}|a|$).
Finally,
in the case where $\beta$ is separating,
we will see the same behavior of
the value $\ell^{\theta}_{2}(a)(|b|)+\ell^{\theta}_{2}(b)(|a|) = \ell^{\theta}_{2}(b)(|a|)$
as in the first case.


 \subsubsection{The case where $\beta$ is non-separating.}\label{sec:non-sepa}
Let $\alpha$, $\beta$ be both non-separating simple closed curves on $\Sigma$.
We take homotopy classes $a$ and $b$ corresponding to the curves arbitrarily.

In the case where $|a|$ and $|b|$ are linearly independent on $H$,
the following holds:
 \begin{thm}\label{thm:non-sep-1}
  %
  Curves $\alpha$, $\beta$ and  $a, b \in \pi$ as above.
  Suppose that $|a|$ and $|b|$ are linearly independent on $H$.
  Let $\theta$ be a symplectic expansion.
  If $i_{G}(\alpha, \beta)=0$, then
  $\ell^{\theta}_{2}(a)(|b|) + \ell^{\theta}_{2}(b)(|a|) \in \mathbb{Z}|a|+\mathbb{Z}|b|$.
  %
  %
 \end{thm}
 \begin{proof}

  Suppose that the curves $\alpha$, $\beta$ satisfy $i_{G}(\alpha, \beta) = 0$.
  While it is not always true that $t_{\beta}(a)=a$ and $t_{\alpha}(b)=b$,
  we may take based loops $a_{0}$ and $b_{0}$ with the base point $p \in \partial\Sigma$
  such that they are freely homotopic to $\alpha$ and $\beta$ respectively and
 satisfy
 \begin{equation*}
  t_{\beta}(a_{0})=a_{0},~
  t_{\alpha}(b_{0})=b_{0},
 \end{equation*}
 because $\alpha$ and $\beta$ are non-separating loops.
 Note that $|a_{0}|=|a|$ and $|b_{0}|=|b|$.
 Taking a symplectic expansion $\theta$
 and applying Lemma \ref{lem:theta_2}, we have
 \begin{equation*}\label{171741_27Nov15}
  \theta_{2}(t_{\beta}(a_{0}))-\theta_{2}(a_{0})
  = |b|\wedge \left\{ \ell^{\theta}_{2}(a_{0})(|b|)+\ell^{\theta}_{2}(b_{0})(|a|) \right\}
  = 0
 \end{equation*}
 and
 \begin{equation*}\label{171751_27Nov15}
  \theta_{2}(t_{\alpha}(b_{0}))-\theta_{2}(b_{0})
   = |a|\wedge \left\{ \ell^{\theta}_{2}(a_{0})(|b|)+\ell^{\theta}_{2}(b_{0})(|a|) \right\}
   = 0.
 \end{equation*}
 It follows from these two equations that
 \begin{equation*}
  \ell^{\theta}_{2}(a_{0})(|b|)+\ell^{\theta}_{2}(b_{0})(|a|) \in \mathbb{Q}|a| \cap \mathbb{Q}|b|.
 \end{equation*}
 Since $|a|$ and $|b|$ are linearly independent,
 $\mathbb{Q}|a| \cap \mathbb{Q}|b| = \{0\}$.
  We then know that
 \begin{equation*}
  \ell^{\theta}_{2}(a_{0})(|b|)+\ell^{\theta}_{2}(b_{0})(|a|) = 0.
 \end{equation*}
 %
  There exist $c, d \in \pi$ such that
  $a_{0} = cac^{-1}$,
  $b_{0} = dbd^{-1}$.
  By using Proposition \ref{prop:ell2}(3) and $|a|{\cdot}|b|=0$,
 \begin{eqnarray*}
  \ell^{\theta}_{2}(a_{0})(|b|)+\ell^{\theta}_{2}(b_{0})(|a|)
   &=& \ell^{\theta}_{2}(cac^{-1})(|b|)+\ell^{\theta}_{2}(dbd^{-1})(|a|) \\
   &=& \left\{ \ell^{\theta}_{2}(a)+|c|\wedge|a| \right\}(|b|)
   +\left\{ \ell^{\theta}_{2}(b)+|d|\wedge|b| \right\}(|a|) \\
  &=& \ell^{\theta}_{2}(a)(|b|)+\ell^{\theta}_{2}(b)(|a|)
   +(|b|{\cdot}|c|)|a|+(|a|{\cdot}|d|)|b|.
 \end{eqnarray*}
  Thus we conclude that
 %
 $\ell^{\theta}_{2}(a)(|b|)+\ell^{\theta}_{2}(b)|a| \in \mathbb{Z}|a|+\mathbb{Z}|b|$.
 \end{proof}

 Next we suppose that $|a|$ and $|b|$ are linearly dependent, i.e., $|b|=k|a|, (k \in \mathbb{Q})$.
 We claim that $k= \pm 1$.
 In fact, since $\beta$ is non-separating,
 we can take a symplectic generators $\{x_{1}, y_{1}, \dots, x_{g}, y_{g}\}$ of $\pi$
 such that $b=x_{1}$, and then we have
 $1 = X_{1}{\cdot}Y_{1} = |b|{\cdot}Y_{1} = k|a|{\cdot}Y_{1}$,
 which proves the claim.

 In this case we choose a symplectic expansion $\theta$ such that
 $\ell^{\theta}_{2}(\pi) \subset \frac{1}{2}{\wedge^{2}}H_{\mathbb{Z}} \subset \mathcal{L}_{2}$,
 where $H_{\mathbb{Z}}=H_{1}(\Sigma;\mathbb{Z})$.
 There exists such a symplectic expansion (see \S 3.3).
 We have the following:

 \begin{thm}\label{thm:non-sep-2}
  Curves $\alpha$, $\beta$, homology classes $a, b$ 
  as above.
  Suppose that $|b|=\pm|a|$,
  If we take a symplectic expansion $\theta$
  such that $\ell^{\theta}_{2}(\pi) \subset \frac{1}{2}{\wedge^{2}}H_{\mathbb{Z}}$,
  then it is always true that $\ell^{\theta}_{2}(a)(|b|)+\ell^{\theta}_{2}(b)(|a|) \in \mathbb{Z}|a|$.
 \end{thm}
 \begin{proof}
  There exists $d \in [\pi, \pi]$ such that $b=a^{\epsilon}d$,
  where $\epsilon = \pm 1$.
 \begin{equation*}
  \ell^{\theta}_{2}(a)(|b|)+\ell^{\theta}_{2}(b)(|a|)
   = \ell^{\theta}_{2}(a)(\epsilon|a|)+\ell^{\theta}_{2}(a^{\epsilon}d)(|a|)
  = 2{\epsilon}\ell^{\theta}_{2}(a)(|a|)+\ell^{\theta}_{2}(d)(|a|). 
 \end{equation*}
  %
  It follows from the condition
  $\ell^{\theta}_{2}(\pi) \subset \frac{1}{2}{\wedge^{2}}H_{\mathbb{Z}}$
  and Proposition \ref{prop:Q|c|}
  that $2\varepsilon\ell^{\theta}_{2}(a)(|a|) \in \mathbb{Z}|a|$.
  Since $\alpha$ and $\beta$ are non-separating,
  we know from Proposition \ref{prop:Q|c|} that
  $\ell^{\theta}_{2}(a)(|a|)$ and
  $\ell^{\theta}_{2}(b)(|b|)$ are both in $\mathbb{Q}|a|=\mathbb{Q}|b|$.
  We can compute $\ell^{\theta}_{2}(b)(|b|)$ as
  \begin{equation*}
   \ell^{\theta}_{2}(b)(|b|)
    = \ell^{\theta}_{2}(a^{\epsilon}d)(\epsilon|a|)
    = \ell^{\theta}_{2}(a)(|a|) + \epsilon\ell^{\theta}_{2}(d)(|a|).
  \end{equation*}
  %
  It follows that $\ell^{\theta}_{2}(d)(|a|) \in \mathbb{Q}|a|$.
  %
  Since $d \in [\pi, \pi]$,
  we have $\ell^{\theta}_{2}(d) \in H_{\mathbb{Z}}\wedge H_{\mathbb{Z}}$
  by Proposition \ref{prop:ell2}(4),
  and
  it indicates $\ell^{\theta}_{2}(d)(|a|) \in H_{\mathbb{Z}}$.
  Therefore we obtain $\ell^{\theta}_{2}(d)(|a|) \in \mathbb{Z}|a|$
  and the proof is done.
%
 \end{proof}
  %
 \subsubsection{The case where $\beta$ is separating}

 \begin{thm}\label{thm:sep}
  Let $\alpha$ be a non-separating simple closed curve
  and $\beta$ a separating simple closed curve on $\Sigma$.
  Homotopy classes $a$ and $b$ for the curves as above.
  Let $\theta$ be a symplectic expansion.
  If $i_{G}(\alpha, \beta)=0$, then
  $\ell^{\theta}_{2}(b)(|a|) \in \mathbb{Z}|a|$.
 \end{thm}
 \begin{proof}
  Suppose $i_{G}(\alpha, \beta) = 0$
  and take a based loop $b_{0}$ satisfying $t_{\alpha}(b_{0})=b_{0}$.
 Note that $|b_{0}|=|b|=0$.
 Taking a symplectic expansion $\theta$
 and applying Lemma \ref{lem:theta_2}, we have
 \begin{equation*}
  \theta_{2}(t_{\alpha}(b_{0}))-\theta_{2}(b_{0})
   = |a|\wedge \left\{ \ell^{\theta}_{2}(b_{0})(|a|) \right\}
   = 0.
 \end{equation*}
 It follows that
  $\ell^{\theta}_{2}(b_{0})(|a|) \in \mathbb{Q}|a|$.
 There exist $d \in \pi$ such that $b_{0} = dbd^{-1}$.
 From Proposition \ref{prop:ell2}(3) and $|b|=0$,
 \begin{equation*}
  \ell^{\theta}_{2}(b_{0})(|a|)
   = \ell^{\theta}_{2}(dbd^{-1})(|a|)
   = \ell^{\theta}_{2}(b)(|a|).
 \end{equation*}
  Thus we have $\ell^{\theta}_{2}(b)(|a|) \in \mathbb{Q}|a|$.
  As in the proof of the previous theorem we can see $\ell^{\theta}_{2}(b)(|a|) \in H_{\mathbb{Z}}$.
  This completes the proof.
 \end{proof}

\subsection{The map $\ell:\pi \rightarrow H\wedge H$}\label{subsec:ell}
With respect to a symplectic generators $\{x_{j}, y_{j}\}_{1 \leq j \leq g}$
of $\pi = \pi_{1}(\Sigma_{g,1}, p)$ ($g \geq 1$),
it is shown \cite{M, Ku} that there exists a symplectic expansion $\theta_{0}$ of $\pi$
satisfying
\begin{equation*}
 \ell^{\theta_{0}}_{2}(x_{j}) = \frac{1}{2}X_{j} \wedge Y_{j}, ~~
  \ell^{\theta_{0}}_{2}(y_{j}) = -\frac{1}{2}X_{j} \wedge Y_{j},
\end{equation*}
for $1 \leq j \leq g$.
%
%
%
We can easily see from Proposition \ref{prop:ell2} that
the map $\ell^{\theta_{0}}_{2}$ is identical with the map $\ell$ defined in \S \ref{sec:intro}.
Note that $\ell(\pi) \subset \frac{1}{2}{\wedge^{2}}H_{\mathbb{Z}}$.
Applying Theorem \ref{thm:non-sep-1}, \ref{thm:non-sep-2} and \ref{thm:sep}
to the symplectic expansion $\theta_{0}$,
we obtain Theorem \ref{thm:main}.





\begin{thebibliography}{00}
 %
\bibitem{FM}
	 B.\ Farb and D.\ Margalit,
	 A primer on mapping class groups,
	 http://www.math.uchicago.edu/margalit/mcg/mcgv31.pdf

 \bibitem{Ka}
	 N.\ Kawazumi,
	 Cohomological aspects of Magnus expansions,
	 arXiv: 0505497[mathGT] (2006)

 \bibitem{KK}
	 N.\ Kawazumi and Y.\ Kuno,
	 The logarithms of Dehn twists, Quantum topology 5, 347-423 (2014)

 \bibitem{Ku}
	 Y.\ Kuno,
	 A combinatorial construction of symplectic expansions,
	 Proc.\ Amer.\ Math.\ Soc.\ 140 no.3, 1075-1083 (2012)

 \bibitem{M}
	 G.\ Massuyeau,
	 Infinitesimal Morita homomorphisms and the tree-level of the LMO invariant,
	 preprint,
	 arXiv:0809.4629 (2008)



\end{thebibliography}
\end{document}